\documentclass[11pt,reqno]{amsart}

\usepackage{amsmath,amssymb,amsfonts}

\usepackage[left=30mm,right=30mm,top=30mm,bottom=30mm]{geometry}
\usepackage{mathrsfs}
\usepackage{hyperref}
\usepackage{cite}
\usepackage[foot]{amsaddr}
\usepackage[sc]{mathpazo} 
\usepackage{graphics}
\usepackage{graphicx}
\usepackage{color}
\usepackage{empheq} 
\usepackage{cases}
\usepackage{enumitem}

\newcommand{\Z}{\mathbb{Z}}
\newcommand{\N}{\mathbb{N}}
\newcommand{\R}{\mathbb{R}}
\newcommand{\C}{\mathbb{C}}
\renewcommand{\L}{\mathsf{L}^2}

\renewcommand{\H}{\mathsf{H}}
\newcommand{\HS}{\mathscr{H}}

\renewcommand{\a}{\mathfrak{a}}
\renewcommand{\aa}{\mathbf{a}}
\newcommand{\A}{{\mathscr{A}}}
\renewcommand{\AA}{{\mathbf{A}}}
\newcommand{\Res}{{\mathscr{R}}}
\newcommand{\J}{\mathscr{J}}

\renewcommand{\rho}{\varrho}

\DeclareMathOperator{\dom}{dom}

\newcommand{\la}{\langle}
\newcommand{\ra}{\rangle}

\newcommand{\eps}{\varepsilon}
\newcommand{\e}{_{\varepsilon}}
\newcommand{\ke}{_{k,\varepsilon}}
\newcommand{\je}{_{j,\varepsilon}}

\newcommand{\al}{\alpha}
\newcommand{\be}{\beta}

\renewcommand{\d}{\,\mathrm{d}}
\newcommand{\ds}{\displaystyle}

\newcommand{\Id}{\mathrm{I}}

\newcommand{\cupl}{\bigcup\limits}
\newcommand{\suml}{\sum\limits}

\newcommand{\wt}{\widetilde}

\newcommand{\ceq}{\coloneqq}

\newcommand{\M}{\mathbb{M}}

\theoremstyle{plain}
\newtheorem{theorem}{Theorem}[section]
\newtheorem*{theorem*}{Theorem}
\newtheorem{lemma}[theorem]{Lemma}
\newtheorem*{lemma*}{Lemma}

\theoremstyle{remark}
\newtheorem{remark}[theorem]{Remark}
\newtheorem*{remark*}{Remark} 

\newtheorem*{example*}{Example} 
\theoremstyle{definition}

\numberwithin{equation}{section}
\numberwithin{figure}{section}

\title
[Creating and controlling band gaps in periodic media with small resonators]
{Creating and controlling band gaps in periodic media with small resonators}

\author[Andrii Khrabustovskyi]{Andrii Khrabustovskyi\,$^{1,2}$}
\address{$^1$ Department of Physics, Faculty of Science, University of
	Hradec Kr\'{a}lov\'{e}, Czech Republic} 
\address{$^2$ Department of Theoretical Physics,
	Nuclear Physics Institute of the Czech Academy of Sciences, \v{R}e\v{z}, Czech Republic} 
\email{andrii.khrabustovskyi@uhk.cz}

\author[Evgen Khruslov]{Evgen Khruslov\,$^{3}$}
\address{$^3$  B.~Verkin Institute for Low Temperature Physics and Engineering of the National Academy of Sciences of Ukraine} 
\email{khruslov@ilt.kharkiv.ua}

\begin{document}

	\begin{abstract}
		We investigate   spectral properties of the Neumann Laplacian $\A\e$ on a periodic unbounded domain $\Omega\e$ depending on a small parameter $\eps>0$. The domain $\Omega\e$  is obtained by removing from $\R^n$ $m\in\N$ families  of $\eps$-periodically distributed small resonators. We prove that the spectrum of $\A\e$ has at least $m$ gaps. The first $m$ gaps
		converge  as $\eps\to 0$ to some intervals whose location and lengths can be controlled by a suitable choice the resonators; other gaps (if any) go to infinity. An application to the theory of photonic crystals is discussed.

	\end{abstract}


 	\maketitle

	\begin{center}\medskip
		\emph{Dedicated to Professor V.A. Marchenko on the occasion of his birthday}
		\medskip
	\end{center}
	
	\section{Introduction}
	
	The problem addressed in this paper belongs
	to	spectral analysis of periodic  differential operators.
	It is well-known (see, e.g., \cite{E73,Ku16,Ku93}) that the spectrum of such operators has the 
	form of a locally finite union of compact intervals (\textit{bands}). 
	In general the bands may touch each other and even (in the multidimensional case) overlap. 
	The open interval $(\al,\be)$ is called a \emph{gap} if it has an empty intersection with the spectrum, but its endpoints belong to it.
	
	The presence of gaps in the spectrum is not guaranteed. For example, the spectrum of the Laplacian on $\mathbb{R}^n$ has no gaps: $\sigma(-\Delta_{\mathbb{R}^n})=[0,\infty)$.
	Therefore the natural and interesting problem arises here: to construct examples of periodic
	operators with non-void spectral gaps. This  problem has
	been actively studied since mid of the 90th and currently a lot of examples for various classes of periodic operators are available in the literature. We refer to some pioneer articles \cite{FK96a,FK96b,G97,HL00,Zh05},
	further references can be found in the overviews \cite{Ku16,HP03}.
	
	The problem of the spectral gaps opening received a strong motivation   coming from the advances in investigation of novel materials of various sorts, in particular, the so-called \textit{photonic crystals} -- periodic dielectric nanostructure whose
	characteristic feature  is that they strongly affect the
	propagation of light waves at certain optical frequencies, which is caused by gaps in the spectrum of the Maxwell operator or related scalar operators.  In practice, one deals with artificially fabricated photonic crystals, which are  composed of relatively simple materials (dielectrics or metals) in a tricky way and have a spatially periodic structure, typically
	at the length scale of hundred nanometers. For more details on mathematics of photonic crystals we refer to   \cite{Dorfler}.
	
	The next important question in spectral theory of periodic operators concerns the possibility of engineering a prescribed gap structure (i.e., opening of spectral gaps with prescribed locations and lengths) --- by choosing a properly devised material texture.
	In the present paper, we investigate this problem for 
	the Neumann Laplacian $\A\e$ on unbounded  periodic domain $\Omega\e\subset\R^n$ ($n\ge 2$),
	which  is obtained by removing from $\R^n$ $m $ families  of $\eps$-periodically distributed small resonators (see Figure~\ref{fig12}, left picture):
	\begin{gather}\label{Omega:eps:intro}
		\Omega\e= \R^n\setminus \overline{\cupl_{i\in\Z^n}\cupl_{j=1}^m\eps ( R\je + i)},\ m\in\N.
	\end{gather}
	Here $\eps>0$ is a small parameter, the sets $R\je$ (\emph{resonators}), $j\in\{1,\dots,m\}$ 
	have the form 
	$$R\je=(F_j\setminus \overline{B_j})\setminus T\je$$
	where $\overline{B_j}\subset F_j\subset (0,1)^n $ and $T\je$ are thin passages connecting the opposite sides of the domain $F_j\setminus \overline{B_j}$ (see Figure~\ref{fig12}, right picture).
	The passages diameters $\eta\je$ are chosen as follows,
	$$\eta\je=\mathcal{O}(\eps^{\frac{2}{n-1}}).$$
	
	\begin{figure}[h]
		\begin{picture}(435,150)	
			\includegraphics[height=150px]{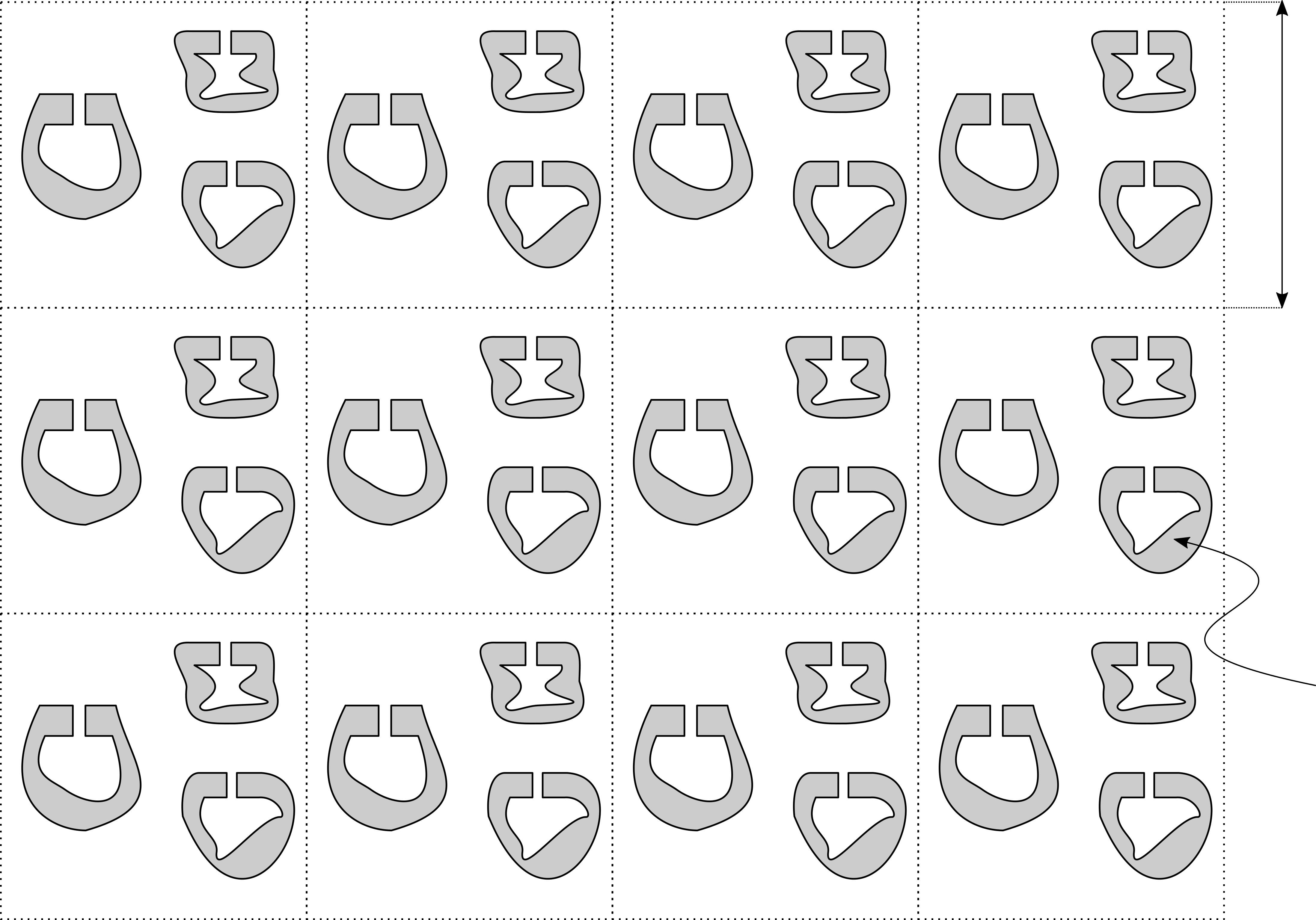}
			\qquad\qquad\qquad\quad
			\includegraphics[height=150pt]{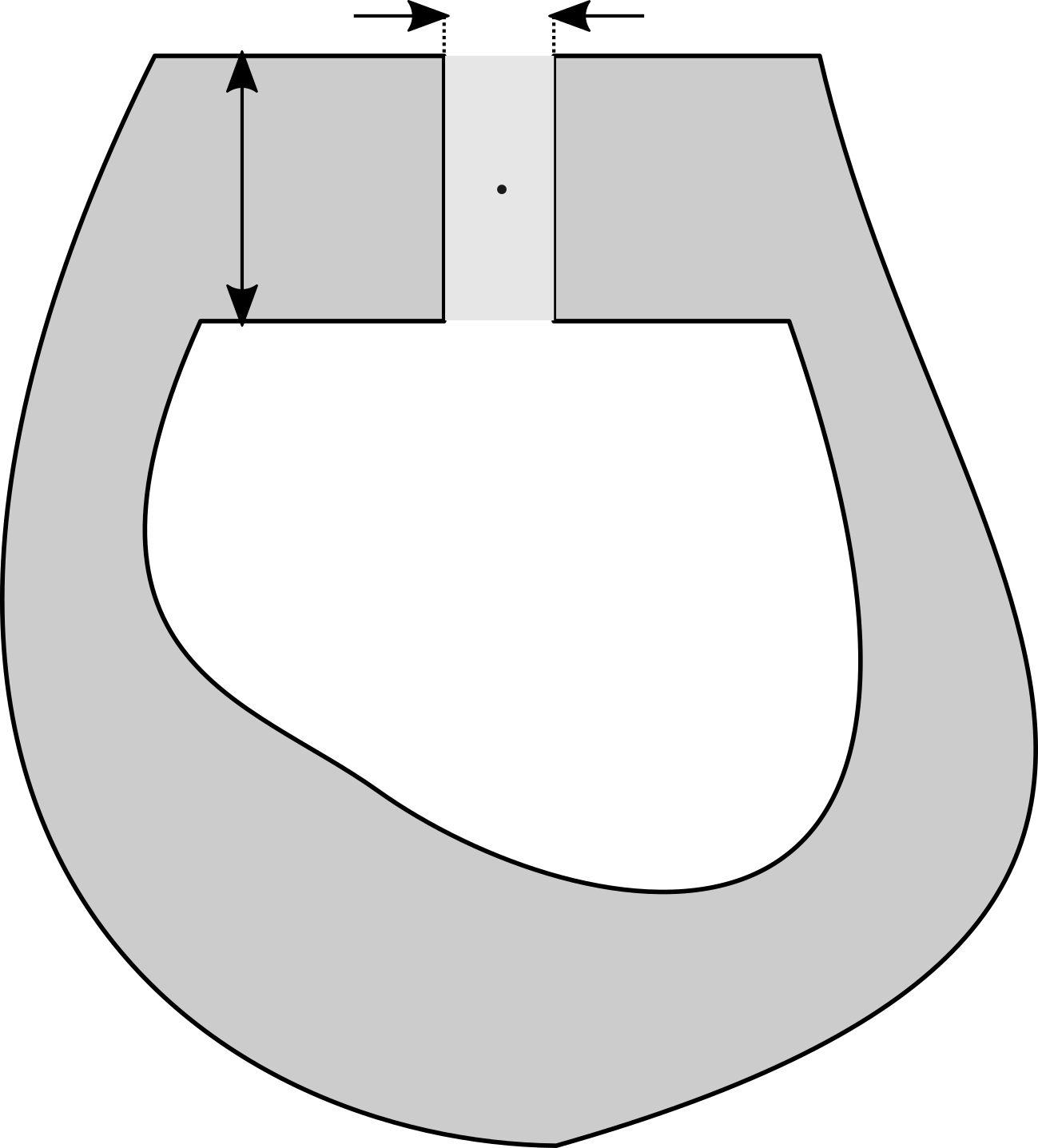}
			
			\put(-213,35){$\eps(R\je+i)$}
			\put(-102,122){$h_j$}
			\put(-78,151){$\eta\je$}
			\put(-219,122){$\eps$}
		\end{picture}
		\caption{\emph{Left:} The domain $\Omega\e$, here $m=3$. The dotted lattice separates different period cells. \emph{Right:} The sets 
			$R\je$ (dark gray color) and $T\je$ (light gray color). The black dot in the center of $T\je$ corresponds to the point $z_j$}\label{fig12}
	\end{figure}	 
	
	We demonstrate that the spectrum $\sigma(\A\e)$ of $\A\e$ has the following properties:
	\begin{itemize}
		\item there is $\Lambda>0$ such that $\sigma(\A\e)$ has $m$ gaps 
		withing the interval $[0,\Lambda\eps^{-2}]$ for sufficiently small $\eps$,
		
		\item the endpoints of these $m$ gaps converges to the endpoints 
		of some pairwise disjoint intervals $(\al_j,\be_j)$. The numbers 
		$\al_j$ and $\be_j$ depend on the lengths   of the passages,  
		the (re-scaled) areas of their cross-sections and the volumes of the domains $F_j$ and $B_j$,    
		
		\item one can choose  the domains $F_j$, $B_j$ and the
		constants standing at the expressions for $\eta\je$ (see \eqref{etaje})  in such a way that
		$(\al_j,\be_j)$ do coincide with predefined intervals.
		
	\end{itemize}\smallskip 
	
	Similar problem was treated by the first author in \cite{Kh14} for ``zero-thickness'' resonators,
	i.e. when $\Omega\e$ is again of the form \eqref{Omega:eps:intro}, but the sets $R\je$ have the form 
	$R\je=\partial B_j\setminus T\je$, where $B_j$ are bounded domains and $T\je$  
	are small subsets of $\partial B_j$. In this case the critical scaling  for the  diameters
	$\eta\je$ of $T\je$ is
	$$
	\eta\je=\mathcal{O}(\eps^{\frac{2}{n-2}})\text{ as }n> 2\qquad\text{and}\qquad
	|\ln(\eps\eta\je)|^{-1}=\mathcal{O}(\eps^{2})\text{ as }n=2.$$

	The problem of opening of spectral gaps having prescribed locations and lengths was also considered in \cite{Kh12} for Laplace-Beltrami operators on periodic Riemannian manifolds, in \cite{Kh13b} for periodic elliptic operators on $\R^n$, in \cite{BK15,Kh20} for periodic quantum graphs, and 
	in \cite{EK18} for periodic Schr\"odinger operators with singular potentials. 
	The proofs in \cite{Kh12,Kh13b} rely on methods of homogenization theory, while 
	in \cite{EK18,BK15,Kh20,Kh14} the approach is close to the one of the present paper (asymptotic analysis of the quasi-periodic, Neumann and Dirichlet  eigenvalue problems on a smallest period cell).
	
	Peculiar effects caused by inserting small resonators are long known.
	For example, a suitably scaled   resonator may abruptly change the spectrum 
	(even if the resonator diameter  is very small) -- the first example goes back to 
	the Courant-Hilbert monograph \cite{CH53}, further investigations were carried out in 
	\cite{AHH91,Sch15}. Another remarkable application of small resonators 
	is the possibility  to construct  materials with frequency-dependent effective properties, with
	large and/or negative permittivities \cite{LS17},  materials with memory \cite{MK06}, etc. (see the overview \cite{Sch17} for more details).\smallskip
	
	In the next section we formulate the problem more precisely and 
	present the main results. We also demonstrate how  to
	apply these results for photonic crystals design.

	\section{Problem setting and main results}\label{sec:2}	
	
	Let $\eps>0$ (the small parameter),
	$n\in\N\setminus\{1\}$ (the space dimension), 
	$m\in\N$ (the number of resonators per a period cell).
	We set 
	\begin{gather}\label{MM}
		\M\ceq \{1,\dots,m\},\quad\M_0\ceq \{0,\dots,m\}. 
	\end{gather}
	In the following, by
	$x'=(x^1,\dots,x^{n-1})$ and $x=(x',x^n)$  we denote the Cartesian coordinates in $\R^{n-1}$ and $\R^{n}$, respectively. \smallskip

	Let   $(F_j)_{j\in\M}$, $(B_j)_{j\in\M}$ be Lipschitz domains in $\R^n$ satisfying
	$$\overline{B_j}\subset F_j,\ \overline{F_j}\subset Y\ceq (0,1)^n,\ j\in\M,\qquad
	\overline{F_i}\cap\overline{F_j}=\emptyset,\ i\not=j.$$ 
	Furthermore, we assume that for each $j\in\M$ there exist 
	$z_j=(z_j',z_j^n)\in \R^n$ and  positive numbers $h_j$ and $d_j$ such that
	\begin{align}\label{pprop1}
		&
		\left\{x=(x',x^n)\in\R^n:\ |x^n-z_j^n|< h_j/2,\  \|x'-z_j'\|_{\R^{n-1}}\leq d_j\right\} \subset F_j\setminus \overline{B_j},
		\\
		\label{pprop2}
		&
		\left\{x=(x',x^n)\in\R^n:\  x^n-z_j^n=   h_j/2,\ \|x'-z_j'\|_{\R^{n-1}}\leq d_j \right\} \subset \partial F_j,
		\\
		\label{pprop3}
		&
		\left\{x=(x',x^n)\in\R^n:\  x^n-z_j^n= - h_j/2,\ \|x'-z_j'\|_{\R^{n-1}}\leq d_j \right\} \subset \partial B_j
	\end{align}
	(here $\|\cdot\|_{\R^{n-1}}$ stands for the Euclidean distance in $\R^{n-1}$).
	We define the passage $T\je$ connecting the opposite sides of $F_j\setminus \overline{B_j}$ via
	\begin{gather}\label{Tke}
		T\je\ceq \left\{x=(x',x^n):\ |x^n-z_j^n|\le h_j/2,\ x'-z_j'\in \eta\je D_j\right\}
		,
	\end{gather}
	where 
	\begin{gather}\label{etaje}
		\eta\je=\eta_j\eps^{2/(n-1)},\ \eta_j>0,
	\end{gather}
	$(D_j)_{j\in\M}$ are connected Lipschitz domains in $\R^{n-1}$ satisfying $0\in D_j$ and
	$\eps$ is sufficiently small in order to have 
	$\eta\je\overline{D_j}\subset\{x'\in\R^{n-1}:\  \|x'\|_{\R^{n-1}}< d_j\}$.
	Finally, we define the domain $\Omega\e$ (see Figure~\ref{fig12}, left picture):
	\begin{gather*}
		\Omega\e= \R^n\setminus \overline{\cupl_{i\in\Z^n}\cupl_{j\in\M}\eps ( R\je+i)}.
	\end{gather*}
	where  the sets $R\je$, $j\in\M$, which will play role of the resonators before scaling (see Figure~\ref{fig12}, right picture), are defined via
	$$R\je=(F_j\setminus \overline{B_j})\setminus\overline{T\je}.$$
	The set $\Omega\e$ is $\Z^n$-periodic with a period cell $\eps Y\e$, where
	$$Y\e\ceq  Y\setminus \overline{\cupl_{j\in\M}R\je}.$$

	Now, we define the (minus) {Neumann Laplacian} $\A\e$ on $\Omega\e$.	 
	In the space $\L(\Omega\e)$ we introduce the sesquilinear form $\a\e$ via
	\begin{gather*}
		\a\e[u,v] =\int_{\Omega\e}\nabla u\cdot\overline{\nabla v} \d x,\quad
		\dom(\a\e)=\H^1 (\Omega\e)
	\end{gather*}
	This form is densely defined, closed, and positive, hence by the first representation theorem \cite[Chapter 6, Theorem 2.1]{K66} there is a the unique self-adjoint and positive operator $\A\e$ satisfying $\dom(\A\e)\subset\dom(\a\e)$ and
	\begin{gather*}
		(\A\e u,v)_{\L(\Omega\e)}= \a\e[u,v],\quad\forall u\in
		\dom(\A\e),\ \forall  v\in \dom(\a\e).
	\end{gather*}
	
	Our goal is to describe the behaviour
	of the spectrum $\sigma(\A\e)$ of $\A\e$ as $\eps\to 0$.
	To state the results we have to introduce some notations.
	
	For $j\in\M$ we denote 
	\begin{gather}\label{Ak}
		\al_j:=\frac{\eta_j^{n-1}  |D_j|}{ h_j  |B_j|}\,,
	\end{gather}
	where the notation $|\cdot|$ stands for the volume of either a domain in $\R^n$ (as $B_j$) or
	a domain in $\R^{n-1}$ (as $D_j$). 
	We assume that the numbers $\al_j$ are pairwise distinct; without loss of generality
	we may assume that 
	\begin{gather}\label{alpha-cond}
		\al_j<\al_{j+1},\ j\in\{1,\dots,m-1\}.
	\end{gather}
	Furthermore, we consider the following   function:
	\begin{gather}\label{mu_eq}
		F(\lambda):=1+\sum_{j\in\M}\frac{\al_j |B_j|}{ |B_0|(\al_j-\lambda)},
	\end{gather}
	where  the set $B_0$ is defined by
	$$
	B_0\ceq Y\setminus\overline{\cup_{j\in\M}F_j}. 
	$$
	It is easy to see that $F(\lambda)$ has exactly $m$ zeros, they are real and interlace with $\al_j$ provided \eqref{alpha-cond} holds. We denote
	them $\be_j$, $j\in\M$ assuming that they are renumbered in the ascending order; then one has
	\begin{gather}\label{inter}
		\al_j<\be_j<\al_{j+1},\quad j\in\{1,\dots,m-1\},\quad
		\al_m<\be_m<\infty.
	\end{gather}

	We are now in position to formulate the main results of this work.
	
	\begin{theorem}\label{th1}
		There exists $\Lambda>0$  depending on the set $B_0$ only such that the
		the spectrum of $\A\e$ has the following form within the interval $[0,\Lambda\eps^{-2}]$ for sufficiently small $\eps$:
		\begin{gather}\label{th1:1}
			\sigma(\A\e)\cap [0,\Lambda\eps^{-2}]=[0,\Lambda\eps^{-2}]\setminus \left(\bigcup\limits_{j\in\M} (\al\je,\be\je)\right)
		\end{gather}
		The closures of the intervals
		$(\al\je,\be\je)\subset (0,\Lambda\eps^{-2})$ are pairwise disjoint and their endpoints satisfy
		\begin{gather}\label{th1:2}
			\lim_{\eps\to 0}\al\je= \al_j ,\quad
			\lim_{\eps\to 0}\be\je= \be_j.
		\end{gather}
	\end{theorem}
	
	Our second result states that  one can choose the resonators in such a way that the limiting intervals $(\al_j,\be_j)$ coincide with prescribed intervals. 
	
	\begin{theorem}\label{th2}
		Let $(\wt \al_j)_{j\in\M}$  and $(\wt \be_j)_{j\in\M}$  be positive numbers satisfying
		\begin{gather}\label{inter+}
			\wt \al_j<\wt \be_j<\wt \al_{j+1},\quad j\in\{1,\dots,m-1\},\quad
			\wt \al_m<\wt \be_m<\infty.
		\end{gather}
		Then one can choose  the domains $F_j$, $B_j$ and the
		numbers $\eta_j$ in such a way that
		\begin{gather*}
			\al_j=\wt \al_j,\quad \be_j=\wt \be_j,\quad j\in\M.
		\end{gather*}	
		
	\end{theorem}

	Before to proceed to the proof of the above results we briefly demonstrate how to 
	apply them for constructing periodic
	$2D$ photonic crystals; for more details see \cite{KK15}. 
	
	We introduce the following sets in $\R^3$ (see Figure~\ref{fig3}):
	$$\wt\Omega\e=\left\{(x^1,x^2,z)\in \R^3:\ x=(x^1,x^2)\in \Omega\e,\ z\in\R\right\},\quad \wt R\e=\R^3\setminus \wt\Omega\e,$$
	where  $\Omega\e\subset \R^2$ is a periodic domain being defined above. 
	We assume
	that $\wt\Omega\e$ is occupied by a dielectric medium
	with the electric
	permittivity and the magnetic permeability being equal to $1$,
	while the set $\wt R\e$ is made of a
	perfectly conducting material.  
	
	\begin{figure}[h]
		\begin{picture}(400,150)
			\includegraphics[height=150px]{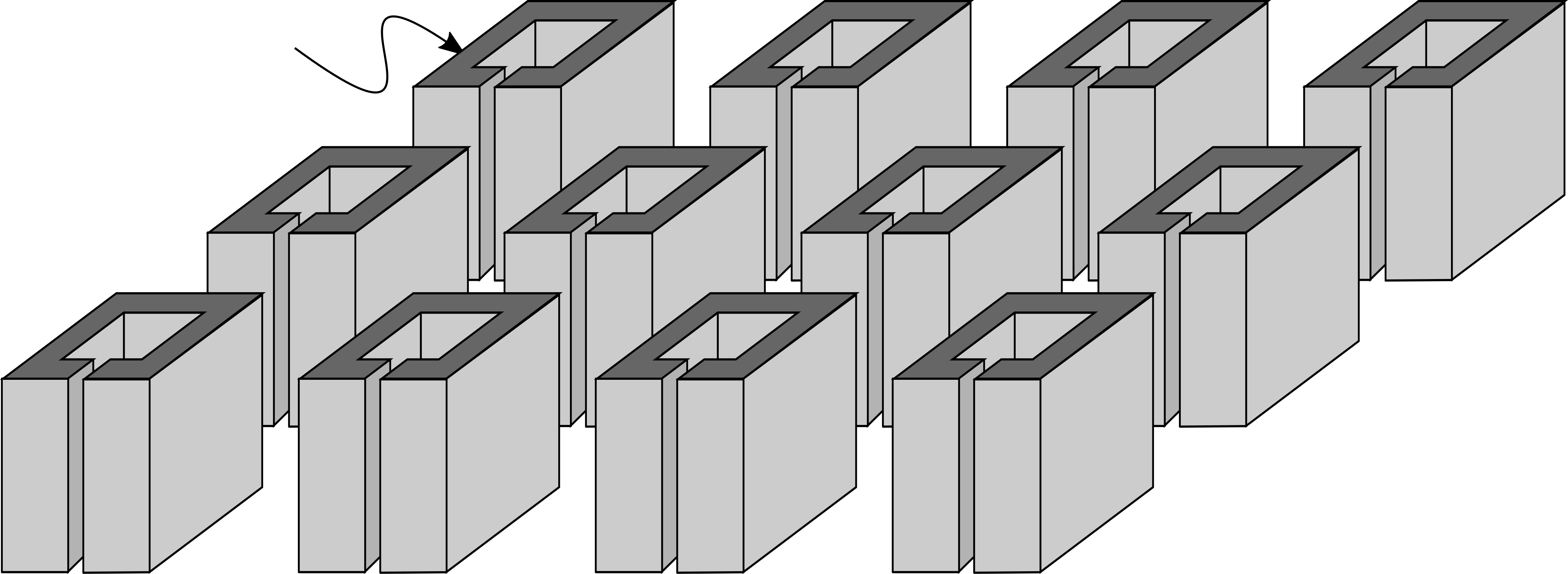}
			\put(-346,135){$\wt R\e$}
		\end{picture}
		\caption{$2D$ photonic crystal. The union of vertical columns $\wt R\e$ is made from a
			perfectly conducting material, while the rest space is occupied by a dielectric}\label{fig3}
	\end{figure}
	
	It is well-known that
	the propagation of electromagnetic waves in the dielectric $\wt\Omega\e$ is governed by the Maxwell operator $\mathscr{M}\e$ acting on $U=(E,H)$ ($E$ and $H$ are the electric and magnetic fields, respectively)
	as follows,
	$$\mathscr{M}\e U=\left( i\hspace{2pt}\nabla\times H,\hspace{2pt} -i\hspace{2pt}\nabla\times E\right),$$
	subject to the   conditions
	\begin{gather*}
		\nabla\cdot E=\nabla\cdot H=0\text{ in }\wt\Omega\e,\quad
		{E}_\tau=0,\ {H}_\nu=0\text{ on }\partial\wt R\e.
	\end{gather*}
	Here ${E}_\tau$ and ${H}_\nu$ are the tangential and normal components of ${E}$ and ${H}$, respectively. 
	
	In the following we focus on the case when $E,H$ depends on $x_1,x_2$ only, i.e. the waves propagated along the plane $\{z=0\}$. It is known that if the medium is periodic in two directions and
	homogeneous with respect to the third one (\emph{$2D$ medium}), then the analysis of the Maxwell operator  reduces to the analysis of scalar elliptic operators on the two-dimensional cross-section $\Omega\e$. Namely, we denote 
	\begin{gather*}
		J=\big\{(E,H):\ \nabla\cdot E=\nabla\cdot H=0\text{ in }\wt\Omega\e,\ E_\tau=H_\mu=0\text{ on }\partial\wt R\e\big\},\\
		J_E=\{(E,H)\in J:\ E_1=E_2=H_3=0\},\quad J_H=\{(E,H)\in J:\ H_1=H_2=E_3=0\}.
	\end{gather*}
	The elements of   $J_E$ and $J_H$ are called TE (Transverse Electric)- and 
	TM(Transverse Magnetic)-polarized waves, respectively.   $J_E$ and $J_H$ are invariant subspaces of $\mathscr{M}\e$, they are $\L$-orthogonal, and each $U\in J$ can be represented in unique way as
	$U=U_E+U_H $ with $U_E\in J_E,\ U_H\in J_H$. Consequently, one has 
	\begin{gather}\label{specM}
		\sigma(\mathscr{M}\e)=\sigma(\mathscr{M}\e|_{J_E})\cup\sigma(\mathscr{M}\e|_{J_H}).
	\end{gather}
	
	We denote by $\mathcal{A}^D\e$ and $\mathcal{A}\e$, respectively, the Dirichlet and the Neumann Laplacians on $\Omega\e$. One can be show  (see, e.g, \cite{HMW61}) that 
	\begin{gather}\label{specM+}
		\omega\in\sigma(\mathscr{M}\e|_{J_E})\ \Leftrightarrow\
		\omega^2\in\sigma(\mathcal{A}^D\e)\quad\text{and}\quad \omega\in\sigma(\mathscr{M}\e|_{J_H})
		\ \Leftrightarrow\ \omega^2\in\sigma(\mathcal{A}\e).
	\end{gather}
	The spectrum of $\A\e$ is already described, see Theorems~\ref{th1}--\ref{th2}.
	As for the spectrum of $\A\e^D$, one can easily derive (cf.~\cite[Lemma 3.1]{KK15}) the 
	following Poincare-type inequality:
	$$\a^D\e[u,u] \geq C\eps^{-2}\|u\|^2_{\L(\Omega\e)},\ \forall u\in\dom (\a^D\e)$$ 
	(the constant $C>0$ is independent of $\eps$, $\a^D\e$ is the form associated with $\A^D\e$).  Consequently, 
	\begin{gather}\label{dir} 
		\inf\sigma(\mathcal{A}^D\e) \to \infty,\ \eps\to 0.
	\end{gather} Then  by virtue of 
	Theorem \ref{th1}, \eqref{specM}--\eqref{dir} we conclude that for an arbitrary 
	(large enough) $L>0$ the Maxwell operator
	$\mathscr{M}\e$ has $2m$ gaps in $[-L,L]$ when $\eps$ is sufficiently small. These gaps converge to the intervals $\pm(\sqrt{\al_j},\sqrt{\be_j})$, whose location and lengths can be controlled via a suitable choice of the resonators (see Theorem~\ref{th2}).
	\smallskip
	
	The rest of the paper is devoted to the proof of the main results. 
	In Section~\ref{sec:3} we prove Theorem~\ref{th1}:
	in Subsection~\ref{subsec:3:1} we sketch some elements of  Floquet-Bloch theory   establishing a relationship between the spectrum of $\A\e$ and the spectra of certain operators on $Y\e$; 
	in Subsection~\ref{subsec:3:2} we detect $\Lambda>0$ such that $\sigma(\A\e)$ has at most $m$ gaps within $[0,\Lambda\eps^{-2}]$; 
	in Subsection~\ref{subsec:3:3} we recall the abstract result from \cite{IOS89} serving 
	to describe the convergence of  eigenvalues of operators in varying Hilbert spaces; 
	using this abstract  result we complete the proof in  Subsections~\ref{subsec:3:4}--\ref{subsec:3:6}. In Section~\ref{sec:4}
	we prove Theorem~\ref{th2}.

	\section{Proof of Theorem~\ref{th1}} \label{sec:3} 
	
	In the following, if $\mathscr{A}$ is a self-adjoint operator with purely discrete spectrum bounded from below and accumulating at $\infty$, we denote by $\lambda_j(\mathscr{A})$ its $k$th eigenvalue, where, as usual, the eigenvalues are arranged in the ascending order and repeated according to their multiplicities. 
	
	By $C,\,C_1,\dots$ we denote generic constants being independent of $\eps$ and of functions appearing in the estimates and equalities where these constants occur.
	
	\subsection{Preliminaries}	\label{subsec:3:1} 
	
	The operator $\A\e$ is $\Z^n$-periodic with the period cell $\eps Y\e$. 
	It is convenient to work further with a period cell $Y\e$, whose internal boundary $\partial Y$ is 
	$\eps$-independent. Thereby, we set 
	$$\Xi\e\ceq\eps^{-1}\Omega\e= \R^n\setminus \overline{\cupl_{i\in\Z^n}\cupl_{j\in\M} (R\je+i)},$$
	and introduce the operator
	$\AA\e$ in $\L(\Xi\e)$ via
	$$ 
	\AA\e=-\eps^{-2}\Delta_{\Xi\e},
	$$
	where $\Delta_{\Xi\e}$ is the Neumann Laplacian on $\Xi\e$.
	The operator $\AA\e$ is periodic with respect to the  period cell $Y\e$, and it is easy to see that
	\begin{gather}\label{sigmaAsigmaA}
	\sigma(\AA\e)=\sigma(\A\e).
	\end{gather}

	The Floquet-Bloch theory (see, e.g., \cite{E73,Ku16,Ku93}) establishes a relationship between the spectrum of 
	$\AA\e$ and the spectra of certain operators on $Y\e$. Namely,  
	let $$\theta=(\theta_1,\dots,\theta_n)\in [0,2\pi)^n.$$ 
	We introduce the space $\H^{1,\theta}(Y\e)$, which consists of functions from $\H^1(Y\e)$ satisfying the following conditions at the opposite faces of $\partial Y$:
	\begin{gather}\label{quasi}
		\forall k\in\{1,\dots,n\}:\  u(x+ e_k)=\exp(i\theta_k) u(x)\  \text{for}\;x=\underset{^{\overset{\qquad\quad\ \uparrow}{\qquad\quad
					k\text{-th place}}\qquad }}{(x_1,x_2,\dots,0,\dots,x_n)},
	\end{gather}
	where $e_k={(0,0,\dots,1,\dots,0)}$.
	In the space $\L(Y\e)$ we introduce the sesquilinear form $\aa^\theta\e$,
	\begin{gather*}
		\aa^\theta\e[u,v]=
		\eps^{-2}\int_{Y\e}\nabla u\cdot \overline{\nabla v} \d x,\quad \dom(\aa^\theta\e)=\H^{1,\theta}(Y\e).
	\end{gather*}
	Let $\AA^\theta\e$ be the associated with this form self-adjoint operator. 
	This operator acts as
	$$-\eps^{-2}\Delta.$$
	The functions $u\in\dom(\AA^\theta\e)$ belong to $\H^2_{\rm loc}(Y\e)$ and, besides \eqref{quasi},
	satisfy
	\begin{gather}\label{quasi+}
		\forall k\in\{1,\dots,n\}:\ \frac{\partial u}{ \partial x_k}(x+ e_k)=\exp(i\theta_k) \frac{\partial u}{ \partial x_k}(x)\  \text{
			for}\;x=\underset{^{\overset{\qquad\quad\uparrow}{\qquad\quad\
					k\text{-th place}}\qquad }}{(x_1,x_2,\dots,0,\dots,x_n)}
	\end{gather}

	The
	spectrum of $\mathbf{A}^\theta\e$ is purely discrete. 
	The Floquet-Bloch theory yields
	\begin{gather}\label{repres1}
		\ds\sigma(\AA\e)=\cupl_{k\in\N} L\ke,
		\text{ where }L\ke\ceq \cup_{\theta\in [0,2\pi)^n}
		\big\{\lambda_j(\AA^\theta\e)\big\},
	\end{gather}
	and moreover, for any fixed $k\in\N$ the set $L\ke$ is a compact interval 
	(\emph{the $k$th spectral band}).

	Along with the operators $\AA\e^\theta$ we also introduce the operators $\mathbf{A}\e^N$ and $\mathbf{A}\e^D$, which differ from $\mathbf{A}\e^\theta$ only by the boundary conditions at $\partial Y$: instead of the $\theta$-conditions we impose   the Neumann and the Dirichlet ones, respectively.
	More precisely, let $\mathbf{A}\e^N$ and $\mathbf{A}\e^D$ be the operators in $\L(Y\e)$ being  	associated with the sesquilinear forms $\aa\e^N$ and $\aa\e^D$ with the domains
	\begin{gather*}
		\aa\e^N[u,v]=\aa\e^D[u,v]=\eps^{-2}\int_{Y\e}\nabla u\cdot \overline{\nabla v} \d x,
		\\	
		\dom(\aa\e^N)=\H^1(Y\e)\text{\quad and\quad }
		\dom(\aa\e^D)=\left\{u\in \H^1(Y\e):\ u\restriction_{\partial Y}=0\right\}
	\end{gather*}
	The spectra of these operators  are purely discrete.  One has
	$$
	\forall\theta\in [0,2\pi)^n:\quad \dom(\aa\e^N)\supset\dom(\aa\e^\theta)\supset\dom(\aa\e^D),
	$$
	whence, by the min-max principle \cite[Section~4.5]{De95}, we get
	\begin{gather}\label{enclosure}
		\forall k\in \mathbb{N},\ \forall\theta\in [0,2\pi)^n:\quad
		\lambda_k(\mathbf{A}\e^N) \leq 
		\lambda_k(\mathbf{A}\e^\theta) \leq
		\lambda_k(\mathbf{A}\e^D).
	\end{gather}
	
	\subsection{Determination of $\Lambda$}\label{subsec:3:2}
	In this subsection we detect $\Lambda>0$ such that the spectrum of  $\A\e$ (or, equivalently, the spectrum of $\AA\e$, cf. \eqref{sigmaAsigmaA}) has at most $m$ gaps in the interval $[0,\Lambda\eps^{-2}]$.
	
	Let   $\theta\in [0,2\pi)^n$. 
	We consider the  following operator in 
	$\L(Y\e)=\L(B_0)\oplus
	\left(\oplus_{j\in\M} \L(B_j\cup T\je)\right)$,
	$$
	\mathbf{A}^{\theta,{\rm dec}}\e= \left(-{ \eps^{-2}}\Delta^{\theta}_{B_0}\right)\oplus
	\left(\oplus_{k=1}^m \left(-{ \eps^{-2}}\Delta_{B_j\cup T\je}\right)\right),
	$$
	where $\Delta^{\theta}_{B_0}$ is the  Laplace operator on $B_0$ subject to the Neumann conditions on $\cup_{j\in \M}\partial F_j$ and conditions \eqref{quasi}, \eqref{quasi+} on $\partial Y$, and $\Delta_{B_j\cup T\je}$ is the Neumann Laplacian on $B_j\cup T\je$, $j\in\M$.
	By $\aa^{\theta,{\rm dec}}\e$ we denote the form associated with $\mathbf{A}^{\theta,{\rm dec}}\e$. In fact, $\mathbf{A}^{\theta,{\rm dec}}\e$ differs from $\mathbf{A}^{\theta}\e$
	by the Neumann decoupling at $\overline{T\je}\cap \overline{B_0}$.

	It is easy to see that 
	\begin{gather*}
		\dom(\aa^{\theta,{\rm dec}}\e)\supset \dom(\aa^{\theta}\e)	
		\quad\text{and}\quad
		\aa^{\theta,{\rm dec}}\e[u,u]= \aa^{\theta}\e[u,u],\ \forall u\in \dom(\aa^{\theta}\e),
	\end{gather*}
	whence, by the min-max principle, we get
	$$
	\forall k\in\M:\quad \lambda_k(\AA^{\theta,{\rm dec}}\e)\leq \lambda_k(\AA^{\theta}\e).
	$$
	The first $m$ eigenvalues of $\mathbf{A}^{\theta,{\rm dec}}\e$ are equal to zero, 
	while the $(m+1)$th eigenvalue equals $\eps^{-2}\Lambda^\theta$,  where $\Lambda^\theta$  is the 
	smallest eigenvalue of the operator $-\Delta^{ \theta}(B_0)$; note that $\Lambda^\theta>0$ if
	$\theta\not=(0,0,\dots,0)$. 	Hence we obtain the estimate
	\begin{gather*}
		\forall \theta\in [0,2\pi)^n:\   \eps^{-2}\Lambda^\theta\leq \lambda_{m+1}(\AA^\theta\e)\leq 
		\sup L_{m+1,\eps}
	\end{gather*}
    (recall that $L_{m+1,\eps}$ is the $(m+1)$th band of $\sigma(\AA\e)$, see \eqref{repres1}),
	whence
	\begin{gather}\label{decoup:est}
		\eps^{-2}\Lambda \leq \sup L_{m+1,\eps},
		\text{ 	where }
		\Lambda:=\max_{\theta\in[0,2\pi)^n}\Lambda^\theta.
	\end{gather}
	Note that $\Lambda$ depends only on the set $B_0$.
	
	From \eqref{decoup:est} and \eqref{repres1} we immediately conclude the following result.
	
	\begin{lemma}\label{lemma:La}
		The spectrum $\sigma(\A\e)$ has at most $m$ gaps within the interval $[0,\Lambda\eps^{-2}]$.
	\end{lemma}

	\subsection{Abstract scheme} \label{subsec:3:3}
	
	To describe the behaviour of the eigenvalues of the operators $\AA\e^N$, $\AA\e^D$ and $\AA\e^\theta$ as $\eps\to 0$, we utilize the abstract result
	from \cite{IOS89} (see also \cite{OIS92} for more detailed proofs) concerning the convergence of  eigenvalues of compact self-adjoint operators in varying Hilbert spaces.   
	\smallskip
	
	Let $\HS\e$ and
	$\HS$ be separable Hilbert spaces, and
	$\Res\e $ and 
	$\Res$ be linear
	compact self-adjoint non-negative operators in $\HS\e$ and $\HS$, respectively.
	We denote by
	$\{\mu_{k,\eps}\}_{k\in\N}$ and
	$\left\{\mu_k\right\}_{k\in\N}$ the sequences of eigenvalues of the
	operators $\Res\e$ and $\Res$, respectively,
	being renumbered in the descending order and with account of their
	multiplicity. 
	
	\begin{theorem}[{\cite[Lemma~1]{IOS89}}]
		\label{th:IOS}
		Let the following conditions $(A_1)-(A_4)$ hold:
		\begin{itemize}
			\item[$(A_1)$] There exists linear bounded operator $\J\e \colon \HS\to
			\HS\e$  such that  
			\begin{gather*}
				\forall f\in\HS:\quad \|\J\e
				f\|_{\HS\e} \to
				\|f\|_{\HS}\text{ as }\eps\to 0.
			\end{gather*}

			\item[$(A_2)$] The operators 
			$ \Res\e  $ are bounded
			uniformly in $\eps$.
			
			\item[$(A_3)$] For any $f\in\HS$ one has $$\|\Res\e \J\e
			f-\J\e \Res f\|_{\HS\e}\to 0 \text{ as }\eps\to 0.$$
			
			\item[$(A_4)$] For any family $\{f\e\in \HS\e\}_\eps$ with $\sup_{\eps}
			\|f\e\|_{\HS\e}<\infty$ there exists a subsequence $(f_{\eps_m})_{m\in\N}$ with $\eps_m\to 0$ as $m\to\infty$ and $w\in
			\HS$ such that $$ \|\Res_{\eps_m} f_{\eps_m}-\J_{\eps_m}
			w\|_{\HS_{\eps_m}}\to 0\text{ as }m\to \infty.$$
		\end{itemize}
		Then for any $k\in\mathbb{N}$ we have
		\begin{equation*}
			\mu_{k,\eps}\to \mu_k\text{ as }\eps\to0.
		\end{equation*}
	\end{theorem}
	
	\begin{remark}\label{rem:IOS}
		The above result was established under the assumption  $\dim \HS=\dim \HS\e=\infty$. Tracing its proof in \cite{IOS89,OIS92} one can easily see  that for $\dim\HS<\infty$ 
		and $\Res$ being a self-adjoint operator in $\HS$ with 
		$$ \sigma(\Res)=\{\mu_1\geq\mu_2\geq\dots\geq \mu_{\dim(\HS)}>0\},
		$$ 
		the result reads as follows:   
		$$\text{conditions }(A_1)\text{-}(A_4)\text{ imply }\lim_{\eps\to 0  }\mu_{k,\eps}=
		\mu_k\text{ for }k\in\{1,\dots,\,\dim \HS\}.$$
	\end{remark}

	\subsection{Asymptotic behavior of Neumann and periodic eigenvalue problems}
	\label{subsec:3:4}
	
	One has:
	\begin{gather}\label{lambda1N}
		\lambda_1(\AA^N\e)=0.
	\end{gather}
	For the next eigenvalues one has the following convergence result.
	
	\begin{lemma}
		\label{lemmaN}
		For any $k\in\{2,\dots,m+1\}$ one has
		\begin{gather}\label{estN}
			\lambda_{k}(\AA^N\e)\to \beta_{k-1},\ \eps\to 0.
		\end{gather}
	\end{lemma}

	\begin{proof}
		Let $\HS^N$ be the  space $\C^{m+1}$ equipped with the weighted scalar product,
		\begin{gather}
			\label{scalar-product}
			(\mathbf{u},\mathbf{v})_{\HS^N}=\suml_{j\in\M_0} u_j \overline{v_j} |B_j|
		\end{gather}
		(recall that the notations $\M$ and $\M_0$ are defined in \eqref{MM}).
		Hereinafter the elements of $\HS^N$ are denoted by bold letters, and their entries are enumerated starting from zero:
		$$
		\mathbf{u}\in \HS^N\ \Rightarrow\ \mathbf{u}=(u_0,\dots,u_m)\text{ with }u_j\in\C.
		$$
		In this space we introduce the sesquilinear form $\aa^N$ via
		$$
		\aa^N[\mathbf{u},\mathbf{v}]=\suml_{j\in\M} \al_j |B_j| (u_j-u_0)\overline{(v_j-v_0)},\quad
		\dom(\aa^N)=\HS^N.
		$$
		Let $\AA^N$ be the operator in $\HS^N$ associated with this form. It is represented by the $(m+1)\times(m+1)$  symmetric (with respect to the scalar product \eqref{scalar-product}) matrix 
		\begin{gather}\label{AN}
			\AA^N=\left(\begin{matrix}\ds\suml_{k=1}^m \al_j |B_j||B_0|^{-1}&-\al_1|B_1||B_0|^{-1}  &-\al_2|B_2||B_0|^{-1}&\dots&-\al_m|B_m||B_0|^{-1} \\-\al_1 &\al_1 &0&\dots&0\\
				-\al_2 &0&\al_2 &\dots&0\\
				\vdots&\vdots&\vdots&\ddots&\vdots\\
				-\al_m  &0&0&\dots&\al_m 
			\end{matrix}\right).
		\end{gather}
		The matrix of the form \eqref{AN} was investigated in \cite{BK15}. It was shown that its eigenvalues 
		$$\lambda_{1}(\AA^N )\leq\lambda_{2}(\AA^N )\leq\dots\leq \lambda_{m+1}(\AA^N )$$
		are the zeros of the function $\lambda F(\lambda)$, where $F(\lambda)$ is defined by \eqref{mu_eq}. Thus, one has
		\begin{gather}\label{lambdaN}
			\lambda_{1}(\AA^N )=0,\quad 
			\lambda_{k}(\AA^N )=\be_{k-1},\ k=2,\dots,m+1.
		\end{gather}	
		Our goal is to show that for $k=1,\dots,m+1$ one has
		\begin{gather}
			\label{conv:lambdaN}
			\lambda_k(\AA^N\e)\to \lambda_k(\AA^N)\text{ as }\eps\to 0.
		\end{gather}
		Then the desired convergence result \eqref{estN}  follows immediately from \eqref{lambdaN}--\eqref{conv:lambdaN}.

		To prove \eqref{conv:lambdaN} we utilize  abstract Theorem~\ref{th:IOS}.	
		We denote 
		$
		\HS\e\ceq \L(Y\e)
		$ and introduce the   operators $\Res^N\e$ and $\Res^N$ acting in $\HS\e$
		and $\HS^N$, respectively:
		$$\Res^N\e\ceq (\AA\e^N+\Id)^{-1},\quad \Res^N \ceq (\AA^N+\Id)^{-1}$$
		(hereinafter $\Id$ stands for an identity operator).  
		The operators $\Res^N\e$, $\Res^N$ are compact, non-negative, moreover, one has
		\begin{gather}\label{A2}
			\|\Res^N\e\|\leq 1.
		\end{gather} 
		We denote by
		$\{\mu_{k,\eps}^N\}_{k\in\N}$  the set of the eigenvalues of $\Res^N\e$ being renumbered in the descending order and with account of their
		multiplicity. By virtue of spectral mapping theorem one has
		\begin{gather}\label{SMT1}
			\mu\ke^N=(\lambda_k(\AA\e^N)+\Id)^{-1},\ k\in\N.
		\end{gather}
		Similarly, we have
		\begin{gather}\label{SMT2}
			\mu_k^N=(\lambda_k(\AA^N)+\Id)^{-1},\ k=1,\dots,m+1,
		\end{gather}
		where $1=\mu_1^N\geq \mu_2^N\ge\dots\ge \mu_{m+1}^N>0$ are the eigenvalues of the operator  $\Res^N$. 
		Finally, we introduce  the linear operator  $\J\e^N:  \HS^N\to \HS\e$ acting on  
		$\mathbf{f}=(f_0,\dots,f_m)$
		as follows,
		$$
		(\J\e^N\mathbf{f})(x)=
		\begin{cases}
			f_j,&x\in B_j,\ j\in\M_0,
			\\
			0,
			&x\in T_{j,\eps},\ j\in\M.
		\end{cases}
		$$
		It is easy to see that for each $\mathbf{f}\in\HS^N $ one has
		\begin{gather}\label{A1}
			\|\J\e^N \mathbf{f}\|_{\HS\e}=\|\mathbf{f}\|_{\HS^N}.
		\end{gather}
		
		Below we demonstrate that the operators $\Res\e^N$, $\Res^N$ and $\J\e^N$
		satisfy the conditions 
		 $(A_3)$ and $(A_4)$
		of Theorem~\ref{th:IOS} (the other two conditions $(A_1)$ and $(A_2)$ are fulfilled due to \eqref{A1} and \eqref{A2}); then, by virtue of Theorem~\ref{th:IOS} and Remark~\ref{rem:IOS}
		we get
		$$\mu_{k,\eps}^N\to \mu_k^N\text{ as }\eps\to 0$$
		for $k=1,\dots,m+1$, whence, owing to \eqref{SMT1}--\eqref{SMT2}, the desired convergence   \eqref{conv:lambdaN} follows.\smallskip

		Let us check the fulfillment of  $(A_3)$.
		Let $\mathbf{f}\in\HS^N$. We denote 
		$f\e\ceq \J\e^N \mathbf{f}$ and 
		\begin{gather}
			\label{ue:def0}
			u\e\ceq \Res^N\e f\e. 
		\end{gather}
		Then $\AA\e^N u\e + u\e = f\e$, whence,
		$u\e\in  \H^1(Y\e)$ and
		\begin{gather}\label{weak:main}
			\aa\e^N [u\e,v\e]+(u\e,v\e)_{\HS\e}=(f\e,v\e)_{\HS\e},\ \forall v\e\in\H^1(Y\e).
		\end{gather}
		The equality \eqref{weak:main} implies easily the estimate
		\begin{gather}\label{apriori:N}
			\eps^{-2}\|\nabla u\e\|_{\L(Y\e)}^2+\|u\e\|_{\L(Y\e)}^2\leq \|f\e\|_{\L(Y\e)}^2=
			\|\mathbf{f}\|_{\HS^N}^2.
		\end{gather}
		In particular, it follows from \eqref{apriori:N} that the norms $\|u\e\|_{\H^1(Y\e)}$ are uniformly
		bounded with respect to  $\eps\in (0,1]$. Hence, by virtue of Banach–Alaoglu and Rellich–Kondrachov theorems,
		there exists a sequence $(\eps_m)_{m\in\N}$ with $\eps_m\searrow 0$ as $m\to\infty$ and 
		$u_j\in \H^1(B_j)$, $j\in\M_0$ such that
		\begin{gather}\label{weak}
			\nabla u_{\eps_m}\to \nabla u_j\text{ weakly in }\L(B_j),
			\\\label{strong}
			u_{\eps_m}\to   u_j\text{ strongly in }\L(B_j),
		\end{gather}
		as $m\to \infty$. 
		Furthermore, using \eqref{apriori:N}, \eqref{weak}, we get  
		$$
		\|\nabla u_j\|^2_{\L(B_j)}\leq 
		\liminf_{m\to\infty}\|\nabla u_{\eps_m}\|_{\L(B_j)}\leq 
		\lim_{m\to\infty}(\eps_m)^2\|\mathbf{f}\|_{\HS^N}=0,$$ 
		whence $u_j$ are constant functions, and we can regard
		$\mathbf{u}=(u_0,\dots,u_{m})$ as the element of $\HS^N$.
		Note that \eqref{strong} implies
		\begin{gather}\label{strong:mean}
			\lim_{m\to\infty}\la u_{\eps_m}\ra_{B_j}= \la u_j\ra_{B_j}=u_j,\ j\in\M_0.
		\end{gather}
		where by $\langle u \rangle_B$ we denote the mean value of the function
		$u(x)$ over the domain $B$, i.e. $$\langle u \rangle_B=\frac{1}{
			|B|}\int_B u(x)\d x.$$
		The same notation will be used further (see \eqref{LHS1}) for the mean value of a function defined on a subset $S$ of an $(n-1)$-dimensional hyperplane, 
		i.e $$\langle
		u\rangle_S=\frac{1}{|S|}\int_{S}u \d s,\quad |S|=\int_{S}\d s,$$
		where  $\d s$ stands for the density of the surface measure on $S$.
		
		Let $\mathbf{v}=(v_0 ,\dots,v_m)\in\HS^N$. We define the function
		$v\e\in\H^1(Y\e)$ via
		\begin{gather*}
			v\e(x)\ceq
			\begin{cases}
				v_j,&x\in B_j,\ j\in\M_0,\\
				\frac{v_0-v_j}{h_j} (x^n-z_j^n) +\frac{v_0+v_j}{2},&x\in T\je,\ j\in\M
			\end{cases}
		\end{gather*}
		(recall that $z_j=(z_j',z_j^n)\in\R^n$ is a point around which we built the passage $T\je$, see \eqref{Tke}).
		Inserting this $v\e$ into \eqref{weak:main}, we arrive at
		the equality
		\begin{gather}\label{weak:ve}
			\eps^{-2}\sum_{j\in\M}\left(\int_{T\je}\frac{\partial u\e}{  \partial x^n}\d x\right) \frac{\overline{v_0 -v_j}}{   h_j}+
			\sum_{j\in\M_0}\la u\e \ra_{B_j} \overline{v_j}|B_j|+\sum_{j\in\M}(u\e, v\e)_{\L(T\je)}=
			(\mathbf{f},\mathbf{v})_{\HS^N}.
		\end{gather}

		Using \eqref{strong:mean} we get  
		\begin{gather}\label{LHS2final}
			\lim_{m\to\infty}\sum_{j\in\M_0}\la u_{\eps_m} \ra_{B_j} \overline{v_j}|B_j|=
			\sum_{j\in\M_0} u_j \overline{v_j}|B_j|=(\mathbf{u},\mathbf{v})_{\HS^N},
		\end{gather}
		Further, one has 
		\begin{gather*} 
			|(u\e, v\e)_{\L(T\je)}|\leq \|u\e\|_{\L(T\je)} \|v\e\|_{\L(T\je)}
			\leq \|u\e\|_{\L(Y\e)}\max\{|v_j|;\,|v_0|\}|T\je|^{1/2},
		\end{gather*}
		whence, taking into account that $\|u\e\|_{\L(Y\e)}\leq  \|f\|_{\HS^N} $ and $|T\je|\to 0$, we conclude
		\begin{gather}\label{LHS3final}
			\lim_{\eps\to 0}\sum_{j\in\M}(u\e, v\e)_{\L(T\je)}=0.
		\end{gather}
		Now, let us inspect the first term in the left-hand-side of \eqref{weak:ve}. 
		We denote by $S\je^\pm$ the top and bottom faces of the passage $T\je$, i.e.,
		$$
		S\je^\pm\ceq \left\{x=(x',x^n)\in\R^n:\ x^n-z_j^n= \pm h_j/2,\ x'-z'_j\in \eta\je D_j\right\}.
		$$ 
		Then, integrating by parts, we obtain:
		\begin{align}\notag
			\eps^{-2}\sum_{j\in\M}\left(\int_{T\je}\frac{\partial u\e}{  \partial x^n}\d x\right)\frac{\overline{v_0 -v_j}}{   h_j}
			&=
			\eps^{-2}\sum_{j\in\M}\left(\int_{S\je^+} u\d s -  \int_{S\je^-} u\d s\right)
			\frac{\overline{v_0-v_j}}{   h_j}\\
			&=
			\sum_{j\in\M}\al_j|B_j|\left(\la u\e\ra_{S\je^+}  -  \la u\e\ra_{S\je^-}  \right)
			(\overline{v_0-v_j} )\label{LHS1}
		\end{align} 
		(on the last step we use $|S\je^\pm|=(\eta\je)^{n-1}|D_j|=(\eta_j)^{n-1}\eps^2|D_j|=\al_j|B_j|h_j\eps^2$, see \eqref{etaje} and \eqref{Ak}).
		One has the following estimates for $j\in\M$:
		\begin{gather}\label{meanest1}
			|\la u\e\ra_{S\je^+}-\la u\e\ra_{B_0}|^2\leq C\|\nabla u\e\|^2_{\L(B_0)}\cdot
			\begin{cases}
				(\eta\ke)^{2-n},&n\ge 3,\\
				|\ln \eta\ke|,&n=2,
			\end{cases}
			\\\label{meanest2}
			|\la u\e\ra_{S\je^-}-\la u\e\ra_{B_j}|^2\leq C\|\nabla u\e\|^2_{\L(B_j)}\cdot
			\begin{cases}
				(\eta\ke)^{2-n},&n\ge 3,\\
				|\ln \eta\ke|,&n=2.
			\end{cases}
		\end{gather}
		The proof is similar to the proof of \cite[Lemma~2.1]{Kh13a}.
		From \eqref{etaje}, \eqref{apriori:N}, \eqref{meanest1}, \eqref{meanest2} we get
		\begin{gather}\label{meanest}
			|\la u\e\ra_{S\je^+}-\la u\e\ra_{B_0}|^2 + |\la u\e\ra_{S\je^-}-\la u\e\ra_{B_j}|^2=\left.
			\begin{cases} 
				\mathcal{O}(\eps^{2/ (n-1)}),&n\ge 3 \\
				\mathcal{O}(\eps^2|\ln\eps |),&n=2 
			\end{cases}	\right\}\to 0\text{ as }\eps\to 0.
		\end{gather}
		Using \eqref{strong:mean} and \eqref{meanest},
		we finally conclude from \eqref{LHS1}:
		\begin{gather}\label{LHS1final}
			\lim_{m\to\infty}(\eps_m)^{-2}\sum_{j\in\M}\left(\int_{T_{j,\eps_m}}\frac{\partial u_{\eps_m}}{  \partial x^n}\d x\right)\frac{\overline{v_0-v_j}}{   h_j}=
			\sum_{j\in\M}{\al_j|B_j|} (u_0-u_j)(\overline{v_0-v_j})=
			\aa^N[\mathbf{u},\mathbf{v}].
		\end{gather}
		Combining \eqref{weak:ve}, \eqref{LHS2final}, \eqref{LHS3final}, \eqref{LHS1final} we 
		arrive at the equality
		\begin{gather*}
			\aa^N[\mathbf{u},\mathbf{v}]+(\mathbf{u},\mathbf{v})_{\HS^N}=(\mathbf{f},\mathbf{v})_{\HS^N},\
			\forall v\in{\HS^N},
		\end{gather*}
		whence we  get
		\begin{gather}\label{u=Resf}
			\mathbf{u}=\Res^N \mathbf{f}.
		\end{gather} 
		The limiting vector $\mathbf{u}$  is independent of the  sequence $u_{\eps_m}$ satisfying \eqref{weak}--\eqref{strong} ($\mathbf{u}$ is defined in a unique way by \eqref{u=Resf}), whence
		we conclude that the whole family $u\e$ converges to $\mathbf{u}$:
		\begin{gather}\label{strong:whole}
			u_{\eps}\to   u_j\text{ strongly in }\L(B_j)\text{ as }\eps\to 0,\ j\in \M_0.
		\end{gather}
		
   	{ 
		Finally, using \eqref{ue:def0}, \eqref{u=Resf} and the definition of the operator $\J\e^N$, we get
		\begin{gather}\label{A3:detailed}
		\|\Res\e^N \J\e^N
		\mathbf{f}-\J\e^N \Res^N \mathbf{f}\|_{\HS\e}^2=
		\sum_{j\in\M_0}\|u\e - u_j\|^2_{\L(B_j)}+\sum_{j\in\M}\|u\e\|^2_{\L(T\je)}.
		\end{gather}
		Due to \eqref{strong:whole} the first term in the right-hand-side of \eqref{A3:detailed}
		tends to zero as $\eps\to 0$. Furthermore, one has the following estimate:
		\begin{gather}\label{T:est}\hspace{-3mm}
		\forall u\in \H^1(T\je\cup B_0):\
		\|u\|^2_{\L(T\je)}\leq
	 C\left(\eta\je^{n-1}\|u\|_{\L(B_0)}^2+\eta\je\kappa\je\|\nabla u\|^2_{\L(B_0)}+\|\nabla u\|^2_{\L(T\je)}\right),
		\end{gather}
		where $\kappa\je\ceq 1$ as $n\ge 3$ and $\kappa\je\ceq |\ln \eta\je|$ as $n=2$.
		The proof of \eqref{T:est} is similar to the proof of inequality (5.16) in \cite{CK17}.
		It follows from \eqref{etaje}, \eqref{apriori:N}, \eqref{T:est} that 
		\begin{gather}\label{T:est:final}
		\|u\e\|^2_{\L(T\je)}\leq 
		C\left(\eta\je^{n-1}\|u\e\|_{\L(Y\e)}^2+\|\nabla u\e\|^2_{\L(Y\e)}\right)\leq
		C_1\eps^2\|\mathbf{f}\|_{\HS^N}^2.
		\end{gather}
	Thus the second term in the right-hand-side of \eqref{A3:detailed}
	goes to zero too; consequently, condition $(A_3)$ is fulfilled.
	}

		\smallskip
		Finally, we check the fulfillment of  $(A_4)$.
		Let $f\e\in\HS\e$ with $\|f\e\|_{\HS\e}\leq C$. We set
		\begin{gather}\label{ue:def}
			u\e\ceq \Res^N\e f\e.
		\end{gather}
		The function $u\e$ belongs to $\H^1(Y\e)$, it satisfies \eqref{weak:main} and 
		the estimate \eqref{apriori:N} holds true. From this estimate  we conclude
		that 
		there exist a sequence $(\eps_m)_{m\in\N}$ with $\eps_m\searrow 0$ as $m\to\infty$ and 
		$w_j\in \H^1(B_j)$, $j\in\M_0$ such that
		\begin{gather}\label{strong+}
			u_{\eps_m}\to   w_j\text{ strongly in }\L(B_j) 
		\end{gather}
		as $m\to \infty$; furthermore, the functions $w_j$ are constants (so, we can regard
		$\mathbf{w}=(w_0,\dots,w_{m})$ as the element of ${\HS^N}$). {
		Also, similarly to \eqref{T:est:final}, we get
		\begin{gather}
		\label{T:est:final+}
		\|u\e\|^2_{\L(T\je)}\leq 
		C\eps^2\|f\e\|_{\HS\e}^2\to0\text{ as }\eps\to 0.	
		\end{gather}
		It follows from \eqref{ue:def}--\eqref{T:est:final+}  and the definition of the operator $\J\e^N$
		that
		$$
		\|\Res_{\eps_m}^N f_{\eps_m}-\J_{\eps_m}^N \mathbf{w}\|_{\HS_{\eps_m}}\to 0\text{ as }m\to\infty.
		$$
		Hence condition $(A_4)$ is also fulfilled.}
		This completes the proof Lemma~\ref{lemmaN}.
	\end{proof}

	Similar result holds for the eigenvalues of the operator 
	$\AA^\theta\e$ with $\theta=(0,0,\dots,0)$, which corresponds to the periodic conditions
	on $\partial Y$.
	One has:
	\begin{gather}\label{lambda1per}
		\lambda_1(\AA^\theta\e)=0\text{ if }\theta=(0,0,\dots,0),
	\end{gather}
	while for the next eigenvalues one has the following lemma.
	
	\begin{lemma}
		\label{lemmaper}
		Let $\theta=(0,0,\dots,0)$. Then
		for any $k\in\{2,\dots,m+1\}$ one has
		\begin{gather*} 
			\lambda_{k}(\AA^\theta\e)\to \beta_{k-1},\ \eps\to 0.
		\end{gather*}
	\end{lemma}	
	
	The proof of Lemma~\ref{lemmaper} repeats verbatim the proof of Lemma~\ref{lemmaN}.
	Note that the boundary conditions \eqref{quasi} imply no 
	restrictions on the limiting constant $u_0$ (see \eqref{weak}--\eqref{strong}), 
	since any constant function satisfies \eqref{quasi} if $\theta=(0,0,\dots,0)$.

	\subsection{Asymptotic behavior of Dirichlet and antiperiodic eigenvalue problems} 
	\label{subsec:3:5}
	
	\begin{lemma}
		\label{lemmaD}
		For any $k\in\{1,\dots,m\}$ one has
		\begin{gather}\label{estD}
			\lambda_{k}(\AA^D\e)\to \al_k,\ \eps\to 0.
		\end{gather}
	\end{lemma}
	
	\begin{proof}
		The proof resembles the one of Lemma~\ref{lemmaN}, thus we underline only the principal 
		differences.
		Let $\HS^D$ be the  space $\C^{m}$ equipped with the  scalar product
		\begin{gather}\label{scalar:D}
			(\mathbf{u},\mathbf{v})_{\HS^D}=\suml_{j\in\M} u_j \overline{v_j} |B_j|
		\end{gather}
		(its elements  are denoted by bold letters, their entries are enumerated  from $1$ to $m$).
		In the space $\HS^D$ we introduce the sesquilinear form $\aa^D$ via
		$$
		\aa^D[\mathbf{u},\mathbf{v}]=\suml_{j\in\M} \al_j |B_j| u_j\overline{v_j},\quad
		\dom(\aa^D)=\HS^D.
		$$
		The associated (with respect to the scalar product \eqref{scalar:D}) operator $\AA^D$ is represented by the $m\times m$  matrix
		$$\AA^D=\mathrm{diag}(\al_1,\dots,\al_m).$$
		Its eigenvalues are given by
		$\lambda_{1}(\AA^D )\leq\lambda_{2}(\AA^D )\leq\dots\leq \lambda_{m }(\AA^D )$, and
		we have
		\begin{gather}\label{lambdaD}
			\lambda_{k}(\AA^D )=\al_k,\ k=1,\dots,m.
		\end{gather}	
		Below we   demonstrate that  for $k=1,\dots,m $ one has
		\begin{gather}
			\label{conv:lambdaD}
			\lambda_k(\AA^D\e)\to \lambda_k(\AA^D)\text{ as }\eps\to 0.
		\end{gather}
		Then the desired convergence result \eqref{estD} follows immediately from \eqref{lambdaD}--\eqref{conv:lambdaD}.

		For the proof of \eqref{conv:lambdaD} we again use   Theorem~\ref{th:IOS}.	
		As before, we denote
		$
		\HS\e\ceq \L(Y\e)$, and introduce compact  non-negative  operators 
		$\Res\e^D\ceq (\AA\e^D+\Id)^{-1}$ and $ \Res^D \ceq (\AA^D+\Id)^{-1}$ acting in $\HS\e$ and $\HS^D$, respectively. 
		One has
		\begin{gather}\label{A2+}
			\|\Res\e^D\|\leq 1.
		\end{gather} 
		We denote by
		$\{\mu_{k,\eps}^D\}_{k\in\N}$  the set of the eigenvalues of $\Res\e^D$ being renumbered in the descending order and with account of their
		multiplicity; similarly, $\mu_1^D\geq \mu_2^D\ge\dots\ge \mu_{m+1}^D$ stand for the eigenvalues of the operator  $\Res^D$. One has
		\begin{gather}\label{SMT12}
			\mu\ke^D=(\lambda_k(\AA\e^D)+\Id)^{-1},\ k\in\N,
			\qquad
			\mu_k^D=(\lambda_k(\AA^D)+\Id)^{-1},\ k=1,\dots,m,
		\end{gather} 
		Finally, we introduce  the   operator  $\J^D\e:  \HS^D\to \HS\e$ acting on  
		$\mathbf{f}=(f_0,\dots,f_m)$
		as follows,
		$$
		(\J^D\e\mathbf{f})(x)=
		\begin{cases}
			f_j,&x\in B_j,\ j\in\M ,
			\\
			0,&x\in B_0\cup\left(\cup_{j\in\M}T\je\right).
		\end{cases}
		$$
		Obviously, for each $\mathbf{f}\in\HS^D $ we have
		\begin{gather}\label{A1+}
			\|\J\e^D \mathbf{f}\|_{\HS\e}=\|\mathbf{f}\|_{\HS^D}.
		\end{gather}
		
		The conditions $(A_1)$ and $(A_2)$ of Theorem~\ref{th:IOS} are fulfilled,  see \eqref{A2+} and \eqref{A1+}. Let us check the fulfillment of the condition $(A_3)$.
		Let $\mathbf{f}\in\HS^D$,
		$f\e\ceq \J\e^D \mathbf{f}$ and 
		\begin{gather}
			\label{ue:def0+}
			u\e\ceq \Res^D\e f\e. 
		\end{gather}
		Then 
		$u\e\in  \H^1(Y\e)$, $u\e=0$ on $\partial Y$, 
		\begin{gather}\label{weak:main:D}
			\aa\e^D [u\e,v\e]+(u\e,v\e)_{\HS\e}=(f\e,v\e)_{\HS\e},\ \forall v\e\in \dom(\aa\e^D).
		\end{gather}
		and the estimate
		\begin{gather}\label{apriori:D}
			\eps^{-2}\|\nabla u\e\|_{\L(Y\e)}^2+\|u\e\|_{\L(Y\e)}^2\leq  
			\|\mathbf{f}\|_{\HS^D}
		\end{gather}
		holds true.
		It follows from \eqref{apriori:D} that the norms $\|u\e\|_{\H^1(Y\e)}$ are uniformly
		bounded with respect to  $\eps\in (0,1]$, whence 
		there exists a sequence $(\eps_m)_{m\in\N}$ with $\eps_m\searrow 0$ as $m\to\infty$ and 
		the constant functions $u_j $, $j\in\M_0$ such that
		\eqref{weak}--\eqref{strong} hold.
		Moreover,  
		$u_{\eps_m}\to u_0$ strongly in $\L(\partial Y)$,
		 hence $u_0=0$ a.e. on $\partial Y$, and, since $u_0$ is a constant function, we conclude
		\begin{gather}
			\label{u0=0}
			u_0\equiv 0.
		\end{gather}
		In the following we regard
		$\mathbf{u}=(u_1,\dots,u_{m})$ as the element of $\HS^D$.
		We fix an arbitrary $\mathbf{v}=(v_1 ,\dots,v_m)\in\HS^D$, and define the function
		$v\e\in\dom(\aa\e^D)$ by
		\begin{gather*}
			v\e(x)\ceq
			\begin{cases}
				v_j,&x\in B_j,\ j\in\M,\\
				0,&x\in B_0,\\
				-\frac{v_j}{   h_j} (x^n-z_j^n) +\frac{v_j}{   2},&x\in T\je,\ j\in\M.
			\end{cases}
		\end{gather*}
		Inserting  $v\e$ into \eqref{weak:main:D}, we obtain the equality
		\begin{gather}\label{weak:ve:D}
			-\eps^{-2}\sum_{j\in\M}\left(\int_{T\je}\frac{\partial u\e}{  \partial x^n}\d x\right) \frac{\overline{   v_j}}{   h_j}+
			\sum_{j\in\M }\la u\e \ra_{B_j} \overline{v_j}|B_j|+\sum_{j\in\M}(u\e, v\e)_{\L(T\je)}=
			(\mathbf{f},\mathbf{v})_{\HS^D}.
		\end{gather}
		Repeating  the arguments from the proof of   Lemma~\ref{lemmaN}
		(taking into account \eqref{u0=0}) we get
		\begin{gather} \label{weak:ve:D:1}
			-\lim_{m\to\infty}(\eps_m)^{-2}\sum_{j\in\M}\left(\int_{T_{j,\eps_m}}\frac{\partial u_{\eps_m}}{  \partial x^n}\d x\right) \frac{\overline{   v_j}}{   h_j}=
			\aa^D[\mathbf{u},\mathbf{v}], 
			\\ \label{weak:ve:D:2}
			\lim_{m\to\infty}\sum_{j\in\M}\la u_{\eps_m} \ra_{B_j} \overline{v_j}|B_j|
			=(\mathbf{u},\mathbf{v})_{\HS^D},
			\quad
			\lim_{\eps\to 0}\sum_{j\in\M}(u\e, v\e)_{\L(T\je)}=0.
		\end{gather}
		From \eqref{weak:ve:D}--\eqref{weak:ve:D:2} we conclude
		$
		\aa^D[\mathbf{u},\mathbf{v}]+(\mathbf{u},\mathbf{v})_{\HS^D}=(\mathbf{f},\mathbf{v})_{\HS^D},\
		\forall v\in{\HS^D},
		$
		whence 
		\begin{gather}\label{u=Resf:D}
			\mathbf{u}=\Res^D \mathbf{f}.
		\end{gather} 
		The limiting vector $\mathbf{u}$  is independent of a  sequence $u_{\eps_m}$, whence
		the whole family $u\e$ converges to $\mathbf{u}$, namely,
		\begin{gather}\label{strong:whole+}
			u_{\eps}\to   u_j\text{ strongly in }\L(B_j)\text{ as }\eps\to 0.
		\end{gather}
		Furthermore (cf.~\eqref{T:est:final}, \eqref{T:est:final+}), one has
		\begin{gather}\label{T:est:final:D}
		\sum_{j\in\M}\|u\e\|^2_{\L(T\je)}\to 0\text{ as }\eps\to 0.
		\end{gather}
		The  property $(A_3)$ follows immediately from \eqref{ue:def0+}, \eqref{u0=0}, \eqref{u=Resf:D},  \eqref{strong:whole+}, \eqref{T:est:final:D} and the definition of the operator $\J\e^D$.
		The property $(A_4)$ is proven similarly (cf.~Lemma~\ref{lemmaN}).
		
		By Theorem~\ref{th:IOS} and Remark~\ref{rem:IOS}  the fulfillment of the conditions 
		$(A_1) -(A_4)$ yield
		$$\mu_{k,\eps}^D\to \mu_k^D\text{ as }\eps\to 0,\ k=1,\dots,m,$$
		whence, using \eqref{SMT12}, we get the desired convergence  \eqref{conv:lambdaD}. Lemma~\ref{lemmaD} is proven.	
	\end{proof}

	\begin{lemma}
		\label{lemmaantiper}
		Let $\theta\notin(0,0,\dots,0)$. Then
		for any $k\in\{1,\dots,m\}$ one has
		\begin{gather*}
			\lambda_{k}(\AA^\theta\e)\to \al_k,\ \eps\to 0.
		\end{gather*}
	\end{lemma}	
	
	The proof of Lemma~\ref{lemmaantiper} is similar to the proof of Lemma~\ref{lemmaD}.
	Here we have to take into account that
	if $u\e$ satisfies  
	the boundary conditions \eqref{quasi}, then \eqref{weak}--\eqref{strong} imply that
	the limiting constant function $u_0$  satisfies \eqref{quasi} too, but for $\theta\notin(0,0,\dots,0)$ this is possible 
	only if $u_0\equiv 0$.
	
	\subsection{End of proof}
	\label{subsec:3:6}
	
	We denote
	$$
	\theta_0\ceq (0,0,\dots,0),\quad \theta_\pi\ceq (\pi,\pi,\dots,\pi).
	$$
	By virtue of \eqref{repres1} and \eqref{enclosure} we get
	\begin{gather}\label{FB+}
		\ds\sigma(\AA\e)=\cupl_{k\in\N} L\ke,
	\end{gather}
	where  the compact intervals $L\ke=[\ell\ke^-,\ell\ke^+]$ satisfy 
	\begin{gather*}
		\lambda_k(\AA^N\e)\leq \ell\ke^- \leq \lambda_k(\AA^{\theta_0}\e),\quad
		\lambda_k(\AA^{\theta_\pi}\e)\leq \ell\ke^+ \leq \lambda_k(\AA^D\e).
	\end{gather*}
	Due to \eqref{lambda1N} and \eqref{lambda1per} one has
	\begin{gather}\label{1}
		\ell_{1,\eps}^-=0.
	\end{gather}	
	Furthermore, Lemmata~\ref{lemmaN}--\ref{lemmaper} yield
	\begin{gather}\label{2}
		\lim_{\eps\to 0}\ell\ke^-= \beta_{k-1},\ k=2,\dots,m+1,
	\end{gather}	
	and  Lemmata~\ref{lemmaD}--\ref{lemmaantiper} give
	\begin{gather}\label{3}
		\lim_{\eps\to 0}\ell\ke^+=\al_{k},\ k=1,\dots,m.
	\end{gather}
	Setting $\al\ke\ceq \ell\ke^+$, $\beta\ke\ceq \ell_{k+1,\eps}^-$	($k=1,\dots,m$)
	we conclude from \eqref{FB+}--\eqref{3}, \eqref{inter} and Lemma~\ref{lemma:La} the desired
	properties \eqref{th1:1}--\eqref{th1:2}.
	Theorem~\ref{th1} is proven.
	
	\section{Proof of Theorem~\ref{th2}}\label{sec:4}
	
	Let $(\wt \al_j)_{j\in\M}$  and $(\wt \be_j)_{j\in\M}$  be positive numbers satisfying
	\eqref{inter+}. 
	We consider the following system of $m$ linear  
	equations with unknowns $\rho_j$, $j=1,\dots,m$:
	\begin{gather}\label{system}
		1+\suml_{j=1}^m \rho_j\frac{\wt\al_j}{   \wt\al_j-\wt\be_k} =0,\quad
		k=1,\dots ,m.
	\end{gather}
	It was shown in \cite[Lemma~4.1]{Kh12} that the system \eqref{system} 
	has the unique solution $\rho_1,\dots,\rho_m$ given by
	\begin{gather}\label{prod}
		\rho_j=\frac{\wt\be_j-\wt\al_j}{  \wt\al_j}\prod\limits_{i=\overline{1,m}|i\not=
			j}\ds\left(\frac{\wt\be_i-\wt\al_j}{    \wt\al_i-\wt\al_j}\right),\ j\in\M.
	\end{gather}
	Note that due to \eqref{inter+} we have
	\begin{gather*}
		\forall j:\ \wt\al_j<\wt\be_j,\quad \forall i\not= j:\
		\mathrm{sign}(\wt\be_i-\wt\al_j)=\mathrm{sign}(\wt\al_i-\wt\al_j)\not= 0
	\end{gather*}
	Consequently, the   numbers $\rho_j$ are positive.
	
	Let $\gamma>0$. We set 
	$$ 
	\tau_j\ceq\frac{\rho_j}{   \ds \gamma^n+\suml_{i\in\M}\rho_i},\ j\in\M,
	$$
	where $\rho_j$ are given by \eqref{prod}.
	One has
	\begin{gather}\label{tau:prop}
		\tau_j>0\text{ and }\sum_{i\in\M}\tau_j<1.
	\end{gather}
	
	Now, let $F_j$, $j=1,\dots,m$ be hyperrectangles with axes being parallel to the coordinate ones and satisfying
	\begin{gather*}
		\cup_{j\in\M}\overline{F_j}\subset Y,\quad \overline{F_i}\cap\overline{F_j}=\emptyset,\ i\not=j,\quad
		|F_j|=\tau_j.
	\end{gather*}
	It is easy to see that such a choice is always possible due to \eqref{tau:prop}.
	Let $y_j\in F_j$ be the center of the hyperrectangle $F_j$; then we define 
	$$B_j\ceq \gamma(F_j-y_j)+y_j$$
	(i.e., $B_j$ is obtained from $F_j$ by a homothety with the center at $y_j$ and the ratio $\gamma$).
	In the following we choose $\gamma<1$, which implies $\overline{B_j}\subset F_j$.
	Further, we remove  from $F_j\setminus\overline{B_j}$ the set $T\je$ 
	of the form \eqref{Tke}; due to the fact that $F_j$ and $B_j$ are two homothetic hyperrectangle with axes being parallel to the coordinate ones, one can always do this  in such a way that the assumptions \eqref{pprop1}--\eqref{pprop3} are fulfilled.

	We choose an arbitrary cross-section profile $D_j$, while the constant $\eta_j$ in \eqref{etaje}
	are chosen as follows:
	\begin{gather*}
		\eta_j=\left(\frac{\wt\al_j	h_j|B_j|}{   |D_j|}\right)^{1/(n-1)}.
	\end{gather*}
	With such a choice of $\eta_j$ we immediately obtain
	\begin{gather}\label{alal}
		\al_j=\wt\al_j,\ j\in\M,
	\end{gather}
	where $\al_j$ are defined by \eqref{Ak}.
	Moreover, we have
	\begin{gather}\label{B/B}
		\frac{|B_j|}{    |B_0|}=
		\frac{\gamma^n|F_j|}{    |Y| -\sum_{i\in\M}|F_i|}
		=
		\frac{\gamma^n \tau_j}{    1 -\sum_{i\in\M}\tau_j}=
		\rho_j.
	\end{gather}
	Since $(\rho_1,\dots,\rho_m)$  is a solution to the system \eqref{system},
	we conclude from \eqref{alal} and \eqref{B/B} that
	\begin{gather*}
		1+\sum_{j\in\M}\frac{\al_j |B_j|}{    |B_0|(\al_j-\wt\be_k)}=0,\ k=1,\dots,m.
	\end{gather*}
	Hence $\wt\beta_j$ are zeros of the function $F(\lambda)$ (see~\eqref{mu_eq}), and consequently
	(taking into account \eqref{inter}, \eqref{inter+}), we obtain
	\begin{gather*}
		\be_j=\wt\be_j,\ j\in\M.
	\end{gather*}
	Theorem~\ref{th2} is proven.

	\section*{Acknowledgment\label{sec:ack}}
	
	The work of A.K. is partly supported by the Czech Science Foundation (GA\v{C}R) through the project 21-07129S.

\end{document}